\documentclass[11pt]{amsart}

\usepackage{amssymb, amsmath}

\newtheorem{theorem}{Theorem}[section]

\newtheorem{corollary}[theorem]{Corollary}

\theoremstyle{definition}
\newtheorem{definition}[theorem]{Definition}

\theoremstyle{remark}
\newtheorem{remark}[theorem]{Remark}

\title{\bf H-Covering Maps in Descriptive Function Theory
}

\author{Alexey Ostrovsky}

\subjclass[2010]{Primary 54E50, 54C10, 26A21; Secondary 54H05.}

\keywords{Descriptive Topological Analysis (DTA), completeness preservation,
stable maps, $H$-covering maps, countable compacta, mapping theory}

\begin{document}

\begin{abstract}

We study the preservation of complete metrizability under maps between metric spaces.
Two main approaches have been developed in this area: one based on the behavior of maps on countable compact sets, and another based on abstract conditions such as stability. In earlier work the author introduced the class of $H$-covering maps and proved that every such map is stable.

The main result of this paper is the converse: a map is $H$-covering if and only if it is stable. This provides a concrete and practical description of stable maps and offers a unified framework for completeness-preserving mappings.
\end{abstract}

\maketitle

\section*{Introduction}
Throughout the paper all spaces are metric, and the reader may think of them as 
subspaces of the real line or of the unit interval $\mathbf{I}=[0,1]$. 
All maps are assumed to be continuous and onto.

The study of completeness preservation under continuous maps goes back to the classical Lavrentiev--Luzin theorem on extending homeomorphisms; see~\cite{Ost} for a broader historical account. Over time two principal lines of research on completeness-preserving maps have developed.

The first line focuses on the behavior of maps on countable compact sets. Its main representatives include the (countable)-compact or $s$-covering maps, harmonic maps, and several related classes. These maps have a concrete geometric nature and are defined in terms of well-understood countable compact sets.

The second line consists of more abstract generalizations intended to describe completeness preservation in structural terms. The most prominent examples are the stable maps and the triquotient maps. Their definitions are abstract and technically involved, and their geometric meaning is less transparent.

A natural question was whether the abstract classes from the second line admit concrete equivalents among the classes of the first line. Several pairwise comparisons were attempted --- for instance, in~\cite{Pr} ---but none of them produced equivalences; instead, they generated new open problems.

In earlier work the author introduced the class of $H$-covering maps, defined inductively using homeomorphisms onto suitable countable compact sets. This gives them a simple and workable structure. Every $H$-covering map is stable, but the converse implication remained open.

The purpose of this paper is to prove that a map is $H$-covering if and only if it is stable. This equivalence shows that stability, originally defined in abstract terms, admits a fully concrete characterization via countable compact sets.

\section{Stability}

\begin{definition}
A map $f : X \to Y$ is called \emph{stable} if to each point $y \in Y$
one can assign a nonempty family $\eta_y$ of subsets of $X$ such that
every $U \in \eta_y$ satisfies $U \cap f^{-1}(y) \neq \emptyset$, and the
following stability condition holds:

For every $U \in \eta_y$ and every open set $V \supset U \cap f^{-1}(y)$,
there exists a neighborhood $O(y)$ of $y$ such that
$V \in \eta_w$ for every $w \in O(y)$.
\end{definition}

Let $y \in Y$ and let $x \in f^{-1}(y)$. By the definition of stability,
the set $X \setminus \{x\}$ belongs to $\eta_w$ for all points $w$
sufficiently close to $y$. Consequently, if $y$ is a non-isolated point,
then there exist points $w$ for which $\eta_w$ contains elements
different from the whole space $X$.

An equivalent definition of stable maps was given by E.~Michael \cite{Mich}.

\begin{remark}
\leavevmode

\begin{enumerate}
\item[\textbf{(I)}]
For every $y \in Y$ we have $X \in \eta_y$.
Indeed, take any $U \in \eta_y$. Since $X \supset U \cap f^{-1}(y)$,
the stability condition  implies a neighborhood $O(y)$ such that
$X \in \eta_w$ for all $w \in O(y)$.

\item[\textbf{(II)}]
For every subset $Y_0 \subset Y$, the restriction

\[
g = f|_{f^{-1}(Y_0)}
\]

is a stable map. A suitable stable family is given by

\[
\eta^{g}_{y} = \{\, U \cap f^{-1}(Y_0) : U \in \eta_y \,\}.
\]

\item[\textbf{(III)}]
Let $U$ be an open subset of $X$ such that $U \supset f^{-1}(y)$.
Then there exists a neighborhood $W(y)$ of $y$, chosen arbitrarily small,
such that

\[
g = f|_{U \cap f^{-1}(W(y))}
\]

is a stable map onto $W(y)$. \begin{proof}
Since $U \supset X \cap f^{-1}(y)$, the stability condition yields a neighborhood
$O(y)$ such that $U \in \eta_w$ for all $w \in O(y)$.
By \textbf{(II)}, the restriction $f|{f^{-1}(O(y))}$ is stable, and hence so is its
further restriction to any open neighborhood $W(y) \subset O(y)$.
\end{proof}
\end{enumerate}
\end{remark}
\section{Canonical compact sets $S_k(y)$}

 Let $N^{+}$ denote the set of positive integers.

Define in $\mathbf{I} = [0,1]$ the sets

\[
\mathbf{D}_0 = \{0\}, \qquad 
\mathbf{D}_1 = \{1/2^{n_1} : n_1 \in N^{+}\},
\]

\[
\mathbf{D}_2 = \{1/2^{n_1+n_2} : n_1,n_2 \in N^{+}\},
\]

and in general

\[
\mathbf{D}_k = \{1/2^{n_1+\cdots+n_k} : n_1,\dots,n_k \in N^{+}\}.
\]

For each $k \in N^{+}$ define the canonical compact

\[
\mathbf{S}_k = \bigcup_{i=0}^k \mathbf{D}_i.
\]

The compact sets $S_k(y)$ and $S_k^{*}(y)$ considered below in a space $Y$
(for finite $k$) are the standard compact sets homeomorphic to $\mathbf{S}_k$.
For infinite $k$ they will be defined later.  
We assume throughout that $Y$ has no isolated points.

For each point $y \in Y$ fix a sequence

\[
D_y = \{y_i : i \in \mathbb{N}^{+}\}
\]

homeomorphic to $\mathbf{D}_1$, consisting of pairwise distinct points
converging to $y$.

Choose pairwise disjoint neighborhoods $O(y_i)$ with 
$\operatorname{diam}(O(y_i)) < 1/i$ and $y \notin O(y_i)$.

For any infinite subsequence

\[
D_y^{*} = \{y_{i_k} : k \in \mathbb{N}^{+},\ i_1 < i_2 < \cdots\}
\]

define the canonical compacta of rank $1$:

\[
S_1(y) = \{y\} \cup D_y, \qquad 
S_1^{*}(y) = \{y\} \cup D_y^{*}.
\]

\section{Definition of $H_\alpha$-covering maps}

Assume that the canonical compact $S_\beta(y)$ and $S_\beta^{*}(y)$ 
are defined for all $\beta < \alpha$.

\subsection*{Successor step: $\alpha = \beta + 1$}

For each $y_i \in D_y$ choose canonical compact

\[
S_\beta(y_i) \subset O(y_i), \qquad 
S_\beta^{*}(y_i) \subset S_\beta(y_i),
\]

and define

\[
S_{\beta+1}(y) = \{y\} \cup \bigcup_{y_i \in D_y} S_\beta(y_i),
\]

\[
S_{\beta+1}^{*}(y) = \{y\} \cup 
\bigcup_{y_{i_k} \in D_y^{*}} S_\beta^{*}(y_{i_k}).
\]

\subsection*{Limit step}

Let $\{\beta_i\}$ be an increasing sequence with $\beta_i < \alpha$.
Define

\[
S_\alpha(y) = \{y\} \cup \bigcup_{y_i \in D_y} S_{\beta_i}(y_i),
\]

\[
S_\alpha^{*}(y) = \{y\} \cup 
\bigcup_{y_{i_k} \in D_y^{*}} S_{\beta_{i_k}}^{*}(y_{i_k}).
\]

\begin{definition}
A map $f : X \to Y$ is called \emph{$H_\alpha$-covering} if for every 
canonical compact $S_\alpha(y)$ there exist $S_\alpha^{*}(y)$ and a set 
$K \subset X$ such that the restriction

\[
f|_K : K \to S_\alpha^{*}(y)
\]

is a homeomorphism.
\end{definition}

\begin{definition}
A map is called \emph{$H$-covering} if it is $H_\alpha$-covering for every 
$\alpha < \omega_1$.
\end{definition}

\section{Main Theorem}

\begin{theorem}
Every stable map $f : X \to Y$ is $H$-covering.
\end{theorem}

\begin{proof}  
We prove by transfinite induction on $\alpha$ with 
$1 \le \alpha < \omega_1$ that every stable map is 
$H_\alpha$-covering.

\subsection*{Step 1: The case $\alpha = 1$}

Consider the canonical compact

\[
S_1(y) = \{y\} \cup D_y = \{y\} \cup \{y_i : i \in \mathbb{N}\},
\]

and define

\[
Z_0(y_i) = f^{-1}(y_i), \qquad
Z_1(y) = f^{-1}(y) \cap 
\operatorname{cl}_X\Big(\bigcup_i Z_0(y_i)\Big).
\]

We first show that $Z_1(y) \neq \emptyset$.
Assume the contrary and consider the open set

\[
U = X \setminus 
\operatorname{cl}_X\Big(\bigcup_i Z_0(y_i)\Big).
\]

Then $U \supset f^{-1}(y)$.
By Remark~1.2(III), there exists a neighborhood $W(y)$  such that

\[
g = f|_{U \cap f^{-1}(W(y))}
\]

is a stable map onto $W(y)$.

Since $y_i \to y$, choose $y_i \in W(y)$.
Then $g^{-1}(y_i) \subset U$, 
contradicting the definition of $U$.
Thus $Z_1(y) \neq \emptyset$.

Let $x \in Z_1(y)$.
Every neighborhood of $x$ meets infinitely many sets $Z_0(y_{i_k})$ in  points  $x_{i_k} \to x$.
Then the compact set $K = \{x\}  \cup \{x_{i_k}:i_k \in N^+ \}\subset X$  is well defined.  Denote  $S^{*}(y)  = f(K)$. Then

\[
f|_{K} : K \to  S^{*}(y)  
\]

is a homeomorphism. 

\subsection*{Inductive step}

Assume that for every $\beta<\alpha$ every stable map is an $H_\beta$-covering.

Let $S_\alpha(y)$ be given, and define

\[
Z_\alpha(y) = f^{-1}(y) \cap 
\operatorname{cl}_X\Big(\bigcup_{y_i \in D_y} Z_{\beta_i}(y_i)\Big).
\]

We show that $Z_\alpha(y) \neq \emptyset$.
Assume the contrary and consider

\[
U = X \setminus 
\operatorname{cl}_X\Big(\bigcup_{y_i \in D_y} Z_{\beta_i}(y_i)\Big).
\]

Then $U \supset f^{-1}(y)$.
By Remark~1.2(III), there exists a neighborhood $W(y)$ such that

\[
g = f|_{U \cap f^{-1}(W(y))}
\]

is a stable map onto $W(y)$.

Moreover, by the same remark, there exist infinitely many pairwise disjoint sets 
$W(y_i) \subset W(y)$ such that each

\[
g_i = g|_{U \cap g^{-1}(W(y_i))}
\]

is stable.

By the inductive hypothesis,  there exists a compact

\[
K_i \subset g^{-1}(W(y_i)) \subset U
\]

such that

\[
g_i|_{K_i} : K_i \to S_{\beta_i}^{*}(y_i)
\]

is a homeomorphism. 
But then $K_i$ must intersect  $Z_{\beta_i}(y_i)$,
contradicting $K_i \subset U$.
Thus $Z_\alpha(y)\neq \emptyset$.

Every neighborhood of every point $x\in Z_\alpha(y)\cap U$ meets infinitely many sets $Z_{\beta_{i_k}}(y_{i_k})$  in points $x_{i_k} \to x$. This follows from the definition of $Z_\alpha(y)$ and the fact that $x$ lies in the closure of  $\cup Z_{\beta_i}(y_i)$.  Let us consider the above-mentioned  homeomorphisms

\[
g_i|_{K_i} : K_i\to S^*_{\beta_i}(y_i)
\]

It is easy to see that we may choose the sets $K_{i_k}$ so that

\[
\operatorname{diam}(K_{i_k}) \to 0.
\]

Then the compact set

\[
K = \{x\} \cup \bigcup_k K_{i_k}
\]

is well defined, and the restriction

\[
f|_K : K \to S^{*}_{\alpha}(y)
\]

is a homeomorphism.

This completes the induction.
\end{proof}

\section{Consequences}

\begin{corollary}
A map $f : X \to Y$   is  $H$-covering  if and only if it is stable.
\end{corollary}

From  \cite[Proposition~1]{Ost2}  it follows that   every $H$-covering map is transquotient, and hence stable. 
The main theorem shows the converse implication, hence the two notions coincide.

Thus stable maps admit a concrete characterization via countable compact sets.

For triquotient maps no such countable characterization is known.
A broad investigation of completeness preservation by triquotient maps is given in [1].

Since every continuous map defined on a compact space is closed, and hence
stable, we obtain from the properties of stable maps the following simple
corollary.

\begin{corollary}
Every continuous map defined on a compact space is an $H$-covering map.
\end{corollary}

\section{Concluding Remarks}

The notion of an $H$-covering map provides a general framework that
brings together several different directions in the study of completeness
preservation. As an illustration, consider the class of $O\cap F$-measurable
maps, that is, maps whose preimages of open sets are intersections of one
open and one closed set.

Such maps are determined by their behavior on canonical countable
compact sets $S_2(y)$ \cite[Corollary~2]{Ost2}, just as $H_2$-covering maps are.

By [7, Theorem 1], every \(O\cap F\)-measurable map is finitely continuous.
Hence, if \(f\) is a bijection and both \(f\) and \(f^{-1}\) are \(O\cap F\)-measurable,
then \(f\) is a finitely homeomorphism. By the Lavrentiev–Luzin theorem,
if \(X\) is completely metrizable, then so is \(Y\).

However, if we replace the assumption that \(f\) is a bijection and both
\(f\) and \(f^{-1}\) are \(O\cap F\)-measurable by the weaker assumption
that \(f\) is \(n\)-to-one (for some \(n\in\{2,3,\dots\}\)) and the
\(O\cap F\)-measurable map \(f\) sends open sets into \(O\cap F\)-sets,
then it is unknown whether \(f\) is a finitely $H$-covering map. This
difficulty becomes even more pronounced when considering compositions
of such maps.
This
leads naturally to the following question:

\medskip

\noindent\textbf{Question.}
Does every continuous $H$-covering map $f\colon X\to Y$, where
$X,Y\subset \mathbf{I}$, extend to such a map on a completely
metrizable subset of $\mathbf{I}$?


\begin{thebibliography}{99}

\bibitem{N} S. Nedev, J. Pelant, V. Valov,
A non-separable Christensen's theorem and set tri-quotient maps
Topol. Appl. 156 (2009), 1234--1240.


\bibitem{Mich} Michael E.,  Partition-complete spaces and their preservation by tri-quotient  and related maps,
Topol.  Appl. 73 (1996) 121--131.


 \bibitem{Pr}  Ostrovsky A.,
Preservation of complete metrizability by covering maps.
Topol. Appl.
 201 (2016),  269--273.




\bibitem{Ost2}
Ostrovsky A.,
On the widest class of completeness-preserving covering maps,
Topol. Appl. 308 (2022), 107998.


\bibitem{Ost}
Ostrovsky A.,
Stable maps of Borel set,
Topol. Appl. 326 (2023), 108413.

\bibitem{Ostro} Ostrovsky A.,
Generalization of sequences and convergence in metric spaces,
Topol. Appl.,
 171  (2014) 63-70.






\bibitem{Ost9}
Ostrovsky A.,
Finitely continuous functions,
Topol. Appl. 261 (2019), 46--50.



\end{thebibliography}
\end{document}